\UseRawInputEncoding
\documentclass[12pt]{article}

\usepackage{dsfont}
\usepackage{amsfonts,amsmath,amsthm}
\usepackage{algorithm}
\usepackage{algpseudocode}
\usepackage{xcolor}
\usepackage{graphicx}
\usepackage{stmaryrd}        

\usepackage[paper=a4paper,dvips,top=2cm,left=2cm,right=2cm,
foot=1cm,bottom=4cm]{geometry}

\newcommand{\dq}[1]{\hat{#1}}
\newcommand{\vdq}[1]{\hat{\mathbf{#1}}}
\newcommand{\dqm}[1]{\hat{#1}}
\newcommand{\dqset}{\hat{\mathbb Q}}
\newcommand{\udqset}{\hat{\mathbb U}}
\newcommand{\cset}{\mathbb C}

\newcommand{\cgroup}{{\mathbb C}^{\times}}

\newcommand{\ucset}{\mathbb T}

\begin{document}
	\large
	
	\title{Unit Dual Quaternion Directed Graphs, Formation Control and General Weighted Directed Graphs}
	\author{ Liqun Qi\footnote{Department of Applied Mathematics, The Hong Kong Polytechnic University, Hung Hom, Kowloon, Hong Kong.
			({\tt maqilq@polyu.edu.hk}).}
		\and \
		\and \
		Chunfeng Cui\footnote{LMIB of the Ministry of Education, School of Mathematical Sciences, Beihang University, Beijing 100191 China.
			({\tt chunfengcui@buaa.edu.cn}).}
		\and {and \
			Chen Ouyang	\footnote{School of Computer Science and Technology,
				Dongguan University of Technology, Dongguan 523000, China. ({\tt oych26@163.com}).}}
	}
	\date{\today}
	\maketitle
	

	\begin{abstract}
		
		We study the multi-agent formation control problem in a directed graph. The relative configurations are expressed by unit dual quaternions (UDQs).  We call such a weighted directed graph a unit dual quaternion directed graph (UDQDG).   We show that
		a desired relative configuration scheme is reasonable or balanced in a UDQDG if and only if
		the dual quaternion Laplacian is similar to the unweighted Laplacian of the underlying directed graph.
		A direct method and a unit gain graph method are proposed to solve the balance problem of general unit weighted directed graphs.
		We then study the balance problem of general non-unit weighted directed graphs.
		Numerical experiments for UDQDG are reported.

		\medskip


	\end{abstract}

	\renewcommand{\Re}{\mathds{R}}
	\newcommand{\rank}{\mathrm{rank}}
	\newcommand{\X}{\mathcal{X}}
	\newcommand{\A}{\mathcal{A}}
	\newcommand{\I}{\mathcal{I}}
	\newcommand{\B}{\mathcal{B}}
	\newcommand{\C}{\mathcal{C}}
	\newcommand{\OO}{\mathcal{O}}
	\newcommand{\e}{\mathbf{e}}
	\newcommand{\0}{\mathbf{0}}
	\newcommand{\dd}{\mathbf{d}}
	\newcommand{\ii}{\mathbf{i}}
	\newcommand{\jj}{\mathbf{j}}
	\newcommand{\kk}{\mathbf{k}}
	\newcommand{\va}{\mathbf{a}}
	\newcommand{\vb}{\mathbf{b}}
	\newcommand{\vc}{\mathbf{c}}
	\newcommand{\vq}{\mathbf{q}}
	\newcommand{\vg}{\mathbf{g}}
	\newcommand{\pr}{\vec{r}}
	\newcommand{\pc}{\vec{c}}
	\newcommand{\ps}{\vec{s}}
	\newcommand{\pt}{\vec{t}}
	\newcommand{\pu}{\vec{u}}
	\newcommand{\pv}{\vec{v}}
	\newcommand{\pn}{\vec{n}}
	\newcommand{\pp}{\vec{p}}
	\newcommand{\pq}{\vec{q}}
	\newcommand{\pl}{\vec{l}}
	\newcommand{\vt}{\rm{vec}}
	\newcommand{\vx}{\mathbf{x}}
	\newcommand{\vy}{\mathbf{y}}
	\newcommand{\vu}{\mathbf{u}}
	\newcommand{\vv}{\mathbf{v}}
	\newcommand{\y}{\mathbf{y}}
	\newcommand{\vz}{\mathbf{z}}
	\newcommand{\T}{\top}
	\newcommand{\R}{\mathcal{R}}

	\newtheorem{Thm}{Theorem}[section]
	\newtheorem{Def}[Thm]{Definition}
	\newtheorem{Ass}[Thm]{Assumption}
	\newtheorem{Lem}[Thm]{Lemma}
	\newtheorem{Prop}[Thm]{Proposition}
	\newtheorem{Cor}[Thm]{Corollary}
	\newtheorem{example}[Thm]{Example}
	\newtheorem{remark}[Thm]{Remark}
	
	\section{Introduction}
	
	The multi-agent formation control problem considers the formation control problem of $n$ rigid bodies in the 3-D space.  It is a very important research problem in robotics research.  Examples of these $n$ rigid bodies can be autonomous mobile robots, or unmanned aerial vehicles (UAVs), or autonomous underwater vehicles (AUVs), or small satellites.   As the unit dual quaternion is a unified mathematical tool to describe the rigid body movement in the 3-D space, dual quaternion is used as a mathematical tool to study the formation control problem \cite{WHYZ12, WYL12}.  
	In multi-agent formation control, a desired relative configuration scheme may be given.   People need to know if this scheme is reasonable or not such that a feasible solution of configurations of these multi-agents exists.    Recently, by exploring the eigenvalue problem of dual Hermitian matrices, and its link with the unit gain graph theory, Qi and Cui \cite{QC24} proposed a cross-disciplinary approach to solve this relative configuration problem.   Cui, Lu, Qi and Wang \cite{CLQW24} investigated dual quaternion unit gain graphs.  {On the other hand, Chen, Hu, Wang and Guo \cite{CHWG24} investigated the application of dual quaternion matrices and dual quaternion unit gain graphs, to the precise formation flying problem of satellite clusters.}
	
	However, as discussed in \cite{WYL12}, the multi-agent formation control problem used to occur in a directed graph, i.e., there are some directed control relations among those multi-agents.   Hence, in this paper, we consider the multi-agent formation control problem in a directed graph.


	In the next section, some preliminary knowledge about the multi-agent formation control problem in a bidirected graph is given.

	In Section 3, we study the multi-agent formation control problem in a directed graph. {Such a directed graph is weighted by unit dual quaternions.   Thus, we call such a directed graph a unit dual quaternion directed graph (UDQDG).} We say a desired relative configuration scheme is reasonable if there is a desired formation vector {corresponding to} it.  We show that a desired relative configuration scheme is reasonable in a  {UDQDG} if and only if for any cycle in this {UDQDG}, the product of relative configurations of the forward arcs of the cycle, and inverses of relative configurations of the backward arcs of the cycle, is equal to $1$.  We then show that   a  {desired} relative configuration scheme in a  {connected UDQDG} {is reasonable if and only if} the dual quaternion Laplacian is similar to the unweighted Laplacian of the underlying directed graph.
	{Suppose that the desired  relative configuration scheme is reasonable, then} zero is an eigenvalue of these two Laplacian matrices, and {if in addition, the underlying graph has a directed spanning tree, then} a formation vector is a desired formation satisfying this reasonable desired relative configuration scheme if and only if it is a a unit dual quaternion vector and an eigenvector, corresponding to the zero eigenvalue,  of the dual quaternion Laplacian matrix.
	
	Based upon the discussion in Section 3, in Section 4, we present a direct method to test if a desired relative  relative configuration scheme is reasonable or not.   We call it a direct method, to distinguish the method presented in Section 5.
	
	{In} Section 5, we propose a gain graph method to determine {whether} a desired relative configuration scheme on a directed graph is reasonable or not.    This method can be applied to any unit weighted directed graph.   Under a certain condition, a given unit weighted directed graph can be converted to a unit gain graph, and the balance problem of that unit weighted directed graph is equivalent to the balance problem of this unit gain graph.
	
	We consider general non-unit weighted directed graphs in Section 6.

	{
		We present several numerical results for verifying whether the desired relative configuration scheme on a directed graph is reasonable or not in Section 7.
		
		Some final remarks are made in Section 8.}
	
	
	\section{Preliminaries}	
	
	\subsection{Unit Gain Graphs}
	
	Let $G=(V,E)$ be a bidirectional graph, where the vertex set $V =\{1,2,\dots,n\}$, and $(j, i) \in E$ as long as $(i, j) \in E$.   There is no multiple edges or loops.   Then we have $|E| = 2m$.
	The out-degree of a vertex $j$, denoted by $d_j=\text{deg}(j)$, is the number of vertices connecting to $j$.
	
	\medskip
	
	A  gain graph is a triple $\Phi=(G, \Omega,\varphi)$ consisting of an {\sl underlying graph} $G=(V,E)$, the {\sl gain group} $\Omega$ and the {\sl gain function} $\varphi: \vec{E}(G)\rightarrow \Omega$ such that $\varphi(i, j) = \varphi^{-1}(j, i)$.
	
	Examples of the gain group inclue  $\{ 1, -1 \}$, $\cgroup = \{ p \in \cset |  p \not = 0 \}$, and $\ucset = \{ p \in \cset | | p | = 1 \}$, etc.  If the gain group $\Omega$ consists of unit elements, then such a gain graph is called a  unit gain graph.
	If there is no  {confusion},  {then} we   simply denote the gain graph as $\Phi=(G, \varphi)$.
	
	A real unit gain graph is called a signed graph.  As we described earlier,
	there are researches on  complex unit gain graphs and  quaternion unit gain graphs.    In the study of multi-agent formation control, the theory of
	dual quaternion unit gain graphs was developed \cite{CLQW24, QC24}.
	
	A switching function is a function $\zeta: V\rightarrow \Omega$ that switches the  $\Omega$-gain graph $\Phi=(G,\varphi)$ to $\Phi^{\zeta}=(G,\varphi^\zeta)$, where
	\[\varphi^{\zeta}(i, j) = \zeta^{-1}(i)\varphi(i, j)\zeta(j).\]
	In this case, $\Phi$ and $\Phi^{\zeta}$ are switching equivalent, denoted by $\Phi\sim\Phi^{\zeta}$.
	We denote  the switching class of $\varphi$ as $[\Phi]$, which is the set of  gain graphs switching equivalent to $\Phi$.
	\bigskip
	
	The gain of a walk $W=\{ i_1, \dots, i_k \}$ is
	\begin{equation}\label{equ:gain_walk}
		\varphi(W)=\varphi(i_1, i_2)\varphi(i_2, i_3)\dots\varphi(i_{k-1},i_k).
	\end{equation}
	A walk $W$ is neutral if $\varphi(W)=1_{\Omega}$, where $1_{\Omega}$ is the identity of $\Omega$.
	An edge set $S\subseteq E$ is balanced if every {cycle} $C\subseteq S$ is neutral. A subgraph is balanced if  its edge set is balanced.
	
	{Let $\Phi=(G,\Omega, \varphi)$ be a unit gain graph.    Define $a_{ij} = \varphi(i, j)$ if $(i, j) \in E$, and $A_{ij} = 0$ otherwise. {Let  {$A(\Phi) = (a_{ij})$} be the adjacency matrix of $\Phi$, respectively. Define Laplacian matrix $L(\Phi)$ of $\Phi$ by
			$L(\Phi) =   D(\Phi)- A(\Phi)$, where
			$D(\Phi)$ is an $n\times n$ diagonal matrix with each diagonal element being the degree of the corresponding vertex {in its underlying graph
				$G$}.}
		
		Note that the adjacency matrix and Laplacian matrix of a unit gain graph are  Hermitian.   This is important as the dual quaternion Hermitian matrix has a nice spectral structure \cite{QL23}.
		
		{A} potential function is a function $\theta: V\rightarrow \Omega$ such that  for every edge $(i, j) \in E$,
		\begin{equation}\label{equ:potential}
			\varphi(i, j) = \theta(i)^{-1}\theta(j).
		\end{equation}
		We write $(G,1_{\Omega})$ as the $\Omega$-gain graph with all neutral edges.
		{The potential function} $\theta$ is not unique since for {any 
			$q\in\Omega$},  $\tilde\theta(v_i) = {\frac{q}{|q|}}\theta(v_i)$ for all $v_i\in V$ is also a potential function of $\Phi$.	
		
		The following theorem can be deduced from \cite{Za89}.
		
		\begin{Thm} \label{lem:balance}
			Let {$\Phi=(G,\Omega, \varphi)$ be a unit gain} 
			graph. Then the following are equivalent:
			\begin{itemize}
				\item[(i)] $\Phi$ is balanced.
				\item[(ii)] $\Phi\sim(G,1_{\Omega})$.
				\item[(iii)] $\varphi$ has a potential function.
			\end{itemize}
		\end{Thm}
		
		{Let $A$ and $L$ be the adjacency and Laplacian matrices of $\Phi$.
			Recently, it was  shown in \cite{CLQW24} that  if $\Phi$ is balanced, then $A$ and $L$  are similar with the adjacency and Laplacian matrices of the underlying graph $G$, respectively, and
			its spectrum
			$\sigma_A(\Phi)$ consists of $n$ real numbers, while its Laplacian spectrum $\sigma_L(\Phi)$ consists of one zero and $n-1$ positive numbers.
			Furthermore,  there is $\sigma_A(\Phi)=\sigma_A(G)$ and $\sigma_L(\Phi)=\sigma_L(G)$.}
		
		\subsection{Unit Dual Quaternions}
		
		Denote the conjugate of a quaternion $q$ by $q^*$.  Let $\mathbb Q$ be the ring of quaternions.
		A {\bf dual quaternion} $\hat q = q_s + q_d\epsilon$ consists of two quaternions $q_s$, the standard part of $\hat q$, and $q_d$, the dual part of $\hat q$, with $\epsilon$ as the infinitesimal unit, satisfying $\epsilon \not = 0$ but $\epsilon^2 = 0$.    If both $q_s$ and $q_d$ are real numbers, then the dual quaternion $\hat q$ is called a dual number.   The magnitude of a dual quaternion $\hat q = q_s + q_d\epsilon$ is defined  as
		\begin{equation} \label{e1}
			|\hat q| := \left\{ \begin{aligned} |q_s| + {(q_sq_d^*+q_d q_s^*) \over 2|q_s|}\epsilon, & \ {\rm if}\  q_s \not = 0, \\
				|q_d|\epsilon, &  \ {\rm otherwise},
			\end{aligned} \right.
		\end{equation}
		which is a dual number.
		
		A dual quaternion $\hat q = q_s + q_d\epsilon$
		is called a {\bf unit dual quaternion} (UDQ) if $|\hat q| = 1$, i.e.,  $q_s$ is a unit quaternion and $q_sq_d^* + q_dq_s^* =0$. Let $\hat {\mathbb Q}$ be the ring of quaternions.
		The set of unit dual quaternions is denoted as $\hat {\mathbb U}$.
		A UDQ  $\hat q\in\hat {\mathbb U}$ can represent the {\bf movement} of a rigid body as
		\begin{equation} \label{e3}
			\hat q = q_s + {\epsilon \over 2}q_s  p^b,
		\end{equation}
		where {$q_s$} is a unit quaternion, representing the rotation of the rigid body, and $p^b = [0, \pp\,^b]$  and ${\pp\,^b}$ representing the {\bf translation} of the rigid body, and the superscript $b$ relating to the body frame, which is attached to the rigid body.  We may use {$\hat q$} to represent the {\bf configuration} of the rigid body.  Then {$q_s$} represents the {\bf attitude} of the rigid body, and $p^b$  {represents} the {\bf position} of the rigid body.
		
		The set of unit dual quaternions is defined as $\udqset$.
		
		{\subsection{Formation Control in a Bidirectional Graph}}
		
		Consider the formation control problem of $n$ rigid bodies.   These $n$ rigid bodies can be autonomous mobile robots, or unmanned aerial vehicles (UAVs), or autonomous underwater vehicles (AUVs), or small satellites.
		Then these $n$ rigid bodies can be described by a graph $G=(V,E)$ with $n$ vertices and $m$ edges.  For two rigid bodies $i$ and $j$ in $V$, if rigid body $i$ can sense rigid body $j$, then edge $(i, j) \in E$.    {In this subsection, we} assume that any pair of these rigid bodies are mutual visual, i.e., rigid body $i$ can sense rigid body $j$ if and only if rigid body $j$ can sense rigid body $i$.   Furthermore, we assume that $G$ is connected in the sense that for any node pair $i$ and $j$, either $(i, j) \in E$, or there is a path connecting $i$ and $j$ in $G$, i.e., there are nodes $i_i, \dots, i_k \in V$ such that $(i, i_1), \dots, (i_k, j) \in E$.
		
		Suppose that for each $(i, j) \in E$, we have a desired relative configuration  {from rigid body $i$ to $j$ as} $q_{d_{ij}} \in \hat{\mathbb U}$.
		
		We say that the desired relative  {configuration}  $\left\{ q_{d_{ij}} : (i, j) \in E \right\}$ is {\bf reasonable} if and only if there is a desired formation $\vdq q_d\in\hat{\mathbb U}^n$, which satisfies
		\begin{equation} \label{desrel}
			q_{d_{ij}} = q_{d_i}^*q_{d_j},
		\end{equation}
		for all $(i, j) \in E$.
		
		Thus, a meaningful application problem in formation control is to verify  {whether} a given desired relative configuration scheme is reasonable or not.

		Define $\varphi(i, j) = q_{d_{ij}}$ for the formation control problem here.
		Then we may apply unit gain graph {theory.}
		{Denote $\Phi=(G,\udqset,\varphi)$. By Theorem~\ref{lem:balance},     the desired relative configuration  is reasonable if and only if $\Phi$  is balanced.}

		
		
		In \cite{QC24}, Qi and Cui proposed verifying the reasonableness of the desired relative configuration  
		by computing the smallest eigenvalue of the Laplacian matrix.  It was show there that a given desired relative configuration scheme is reasonable or not if and only if the corresponding dual quaternion unit gain graph is balanced.   Furthermore, they proved the following theorem.
		
		\begin{Thm}\label{thm:balanced}
			Let $\Phi=(G,\udqset, \varphi)$ be a dual quaternion unit gain graph with $n=|V|$, $m=|E|$. Suppose that $\varphi(i, j)=1$ for all $(i, j)\in E$ and  $G$ has $t$ subgraphs $G_i=(V_i,E_i)$, $n=n_1+\dots+n_t$, $V_i=\{n_1+\dots+n_{i-1}+1, \dots,n_1+\dots+n_i\}$, and {$n_0=0$.} Let   $L$ be the Laplacian {matrix} of $\Phi$.  Then $\Phi$ is balanced if and only if  the following conditions hold simultaneously:
			
			(i) $L$  has $t$ zero eigenvalues that are the smallest eigenvalue of $L$;
			
			(ii)  their  eigenvectors $\vx_i\in\dqset^n$ satisfies {$|x_i(j)|=\frac{\sqrt{n_i}}{n_i}$} for $j=n_1+\dots+n_{i-1}+1, \dots,n_1+\dots+n_i$,   $|x_i(j)|=0$ otherwise;
			
			(iii) $Y^*LY$ is equal to the Laplacian matrix of the underlying graph $G$, where $Y=\text{diag}(\vy)\in\dqset^{n\times n}$, $\vx_a=\vx_1+\dots+\vx_t$, and $y(i) = \frac{x_a(i)}{|x_a(i)|}$.
		\end{Thm}
		
		Thus, an efficient cross-disciplinary method to test the reasonableness of a multi-agent relative configuration scheme was presented.   
		
		{However, when the rigid bodies are not mutual visual,  it remains uncertain how to verify the reasonableness of the desired relative configuration.
			This motivates us to study unit dual quaternion directed graphs.
		}

		\section{Unit Dual Quaternion Directed Graphs and Formation Control}
		
		In this section, we consider the formation control problem of $n$ rigid bodies  {in} a {directed}  {(digraph)} $G=(V,E)$ with $n$  {nodes} and $m$  {arcs}.  Mathematically, we express this problem as a unit dual quaternion directed graph problem.

		For two rigid bodies $i$ and $j$ in $V$, if rigid body $i$ can sense rigid body $j$, then  {arc} $(i, j) \in E$,  {where $i$ and $j$ are the tail and head of the arc $(i, j)$ respectively}.
		{It is possible that $(i, j) \in E$ but $(j , i) \not \in E$, i.e.,} 
		rigid body $i$ can sense rigid body $j$  but rigid body $j$ cannot sense rigid body $i$.
		
		A digraph $G$ is  {weakly} connected or simply called connected if there is an undirected path between any pair of vertices, or equivalently, the underlying undirected graph of $G$ is   	connected.
		
		{Suppose that there is a cycle $C = \{ i_1, i_2, \dots, i_k, {i_{k+1}} \equiv i_1 \}$ in a directed graph $G$.   Then $C$ contains $k$ arcs in $E$.   Suppose that $1 \le j \le k$.  If arc $\left( i_j, i_{j+1}\right) \in E$ is a member of $C$, then we say that it is a {\bf forward arc} in $C$, and denote that $\left(i_j, i_{j+1}\right) \in C$.   If arc $\left( i_{j+1}, i_j\right) \in E$ is a member of $C$, then we say that it is a {\bf backward arc} in $C$, and denote $\left( i_{j+1}, i_j\right) \in C$.}

		Suppose that for each $(i, j) \in E$, we have a desired relative configuration  {from rigid body $i$ to $j$ as}  {a UDQ} $\hat q_{d_{ij}} \in \hat{\mathbb U}$.   We say that the desired relative configuration scheme $\left\{ \hat q_{d_{ij}} : (i, j) \in E \right\}$ is {\bf reasonable} if and only if there is a desired formation $\vdq q_d\in {\hat{\mathbb U}^n}$, which satisfies
		\begin{equation} \label{desrel}
			\hat q_{d_{ij}} = \hat q_{d_i}^*\hat q_{d_j},
		\end{equation}
		for all $(i, j) \in E$.
		
		We may consider $\hat q_{d_{ij}}$ as a weight of the arc $(i, j)$.   Then we have a UDQ directed graph (UDQDG)  $\Phi = (G, \udqset, \varphi)$, where $G = (V, E)$ and $\varphi(i, j) = \hat q_{d_{ij}}$ for $(i, j) \in E$.
		
		The Laplacian matrix of  a {UDQDG}  {$\Phi$} is  defined by
		\begin{equation}\label{Laplaican_mat}
			\dqm L =   {D-\dqm A},
		\end{equation}
		where  {$D$} is a diagonal real matrix  {whose}  $i$-th diagonal element  is equal to the  {out}-{degree $d_i$ (the number of  {arcs  {going} out from} $i$)} of  the $i$-th vertex:
		and ${\dqm A}=(\dq a_{ij})$,
		\begin{equation*}
			\dq a_{ij} = \left\{
			\begin{array}{cl}
				{\hat q_{d_{ij}}},  &  \text{if } (i,j)\in E,\\
				0, & \text{otherwise},
			\end{array}
			\right.
		\end{equation*}
		where $\hat q_{d_{ij}}$ is the desired relative configuration on the edge $(i, j) \in E$.   Now
		for $(i, j) \in E$, let $\dq a_{ij}\in\hat{\mathbb U}$ be the desired relative configuration of the rigid bodies $i$ and $j$. Then {$\dqm A$} is
		the desired relative configuration adjacency matrix.

		Next, we present an example.

		\begin{example} \label{ex:2.1} Consider the two figures in Fig.~\ref{fig:DST3}. For the  the directed  tree on the left,  we have {$G^{(1)} = (V^{(1)}, E^{(1)})$,  $V^{(1)}=\{1,2,3\}$, $E^{(1)}=\{(2,1), (3,1)\}$,} and
			\begin{equation}\label{example:DST}
				\hat A^{(1)} = \begin{bmatrix}
					0 & 0 & 0 \\
					\hat q_{d_{21}} & 0 & 0\\
					\hat q_{d_{31}}& 0 & 0
				\end{bmatrix},\quad  	\hat L^{(1)} = \begin{bmatrix}
					0 & 0 & 0 \\
					-\hat q_{d_{21}} & 1 & 0\\
					-\hat q_{d_{31}}& 0 & 1
				\end{bmatrix}.
			\end{equation}
			For the  directed cycle on the right,  we have {$G^{(2)} = (V^{(2)}, E^{(2)})$,  $V^{(2)}=\{1,2,3\}$, $E^{(2)}=\{(1,3),(2,1), (3,2)\}$,} and
			\begin{equation}\label{example:cycle}
				\hat A^{(2)} = \begin{bmatrix}
					0 & 0 & \hat q_{d_{13}} \\
					\hat q_{d_{21}} & 0 & 0\\
					0 &\hat q_{d_{32}} & 0
				\end{bmatrix},\quad  	\hat L^{(2)} = \begin{bmatrix}
					1 & 0 & -\hat q_{d_{13}} \\
					-\hat q_{d_{21}} & 1 & 0\\
					0 &-\hat q_{d_{32}} & 1
				\end{bmatrix}.
			\end{equation}
		\end{example}
		
		\begin{figure}\label{fig:DST3}
			\begin{center}
				\includegraphics[width=0.3\linewidth]{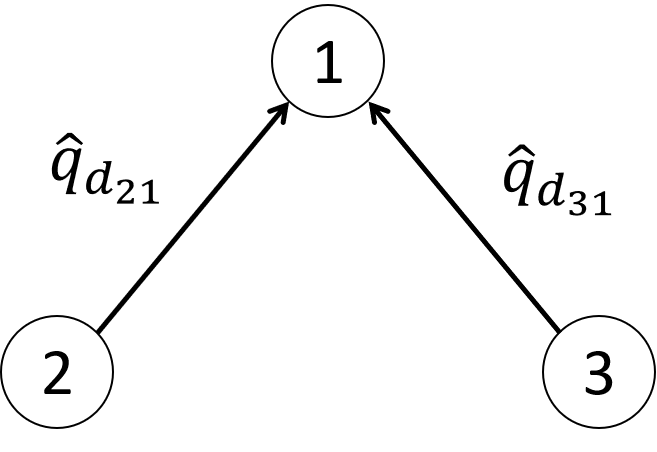} \qquad
				\includegraphics[width=0.3\linewidth]{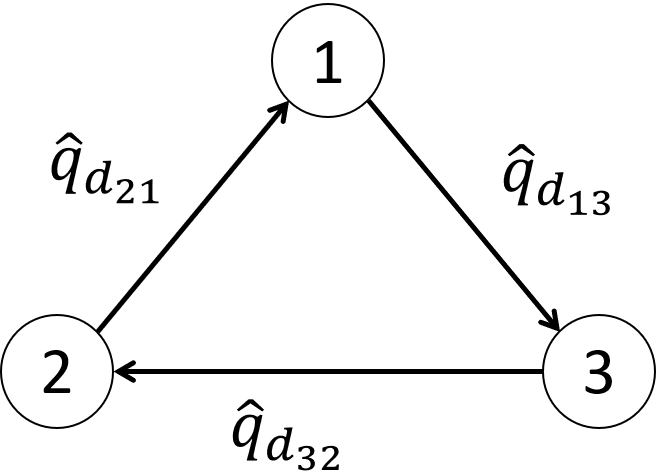}
			\end{center}
			\caption{Two directed graphs. Left: A directed   tree with three vertices. Right: A directed cycle with three vertices.}
		\end{figure}

		Note that,  for a directed graph, neither  {$\dqm A$}  {nor} {$\dqm L$} is a dual quaternion Hermitian matrix. In this case, their spectral analysis {is}  complicated.

		The spectral theory of the Laplacian matrix of a directed graph {$G$} is important.  Let $L$ be the unweighted Laplacian  matrix of $G$.
		The following result from \cite{RB08} is quite useful in our analysis.
		\begin{Lem}[\cite{RB08}, Corollary 2.5] \label{Lem:zeroeig_G}
			The nonsymmetrical Laplacian matrix  of a directed graph has a simple zero eigenvalue with an associated eigenvector ${\bf 1}$  and all of the other eigenvalues are in the open right half plane if and only if the directed graph has a directed spanning tree.
		\end{Lem}
		{Here, $G$ has a directed spanning tree if there exists a node $r$ (a leaf) such that all
			other nodes can be linked to $r$ via a directed path.
			In \cite{RB08}, $r$ is a root since the entries of the adjacency matrix are nonzero if $(j,i)\in E$, which is different from ours. }
		
		{Note that Lemma~\ref{Lem:zeroeig_G} may not hold for Laplacian matrices of UDQDGs. For instance, let $\hat q_{d_{13}}=[0,1,0,0]$,  $\hat q_{d_{21}}=[0,0,1,0]$, and  $\hat q_{d_{32}}=[0,0,0,1]$ in \eqref{example:cycle}, then  we have $\mathrm{rank}(\hat L^{(2)})=2$. If we change  $\hat q_{d_{32}}$ to any other unit dual quaternion numbers not equal $[0,0,0,1]$, then we have $\mathrm{rank}(\hat L^{(2)})=3$.
		}

		{By \eqref{desrel}, we have the following result.}
		\begin{Lem}
			{Assume that} the desired relative configuration $\left\{ \hat q_{d_{ij}} : (i, j) \in E \right\}$ is reasonable.   Suppose both $(i, j)$ and $(j, i)$ are in $E$.   Then we have
			\begin{equation} \label{symmetry}
				\hat q_{d_{ij}} = \hat q_{d_{ji}}^*.
			\end{equation}
		\end{Lem}
		{\begin{proof}
				This result could be seen directly by \eqref{desrel}.
		\end{proof}}

		{The above lemma is only a necessary condition for the desired relative configuration scheme being reasonable.
			Next, we show an equivalent condition in the following theorem.}
		
		\begin{Thm}  \label{Reasonable} Suppose that we have a {UDQDG} {$\Phi = (G, \udqset, \varphi)$, where} $G = (V, E)$ {that has a directed spanning tree.}
			The desired relative configuration $\left\{ \hat q_{d_{ij}} {\equiv \varphi(i, j)}: (i, j) \in E \right\}$ is {\bf reasonable} if and only if
			for any cycle $C = \left\{ i_1, \dots, i_k, i_{k+1} \equiv i_1 \right\}$ of $G$, we have
			\begin{equation} \label{cycle}
				\prod_{1 \le j \le k} \hat p_{i_ji_{j+1}} = 1,
			\end{equation}
			where $\hat p_{i_ji_{j+1}} = \hat q_{d_{i_ji_{j+1}}}$ if $\left(i_j, i_{j+1}\right) \in E$, and
			$\hat p_{i_ji_{j+1}} = \hat q_{d_{i_ji_{j+1}}}^*$ if $\left(i_{j+1}, i_j\right) \in E$.
			
			Furthermore, let {${\vdq q}_0$} be a desired formation satisfying (\ref{desrel}).  Then the set of all the desired formations satisfying (\ref{desrel}) is
			\begin{equation} \label{alldesiredformation}
				{[\vdq q_0]=}\left\{ \hat c \vdq q_0 : \hat c \in \hat{\mathbb U} \right\}.
			\end{equation}
		\end{Thm}
		\begin{proof}   Suppose that the desired relative configurations $\left\{ \hat q_{d_{ij}} : (i, j) \in E \right\}$ is reasonable.   Then there is a desired formation $\vdq q_d\in\hat{\mathbb U}^{n}$, which satisfies (\ref{desrel}) for all $(i, j) \in E$.  Then this implies (\ref{cycle}).
			
			On the other hand, suppose that the desired relative configurations {$\left\{ \hat q_{d_{ij}} \right\}$} satisfy (\ref{cycle}) for all cycles in $G$.   {Let $m$ be the number of edges of the underlying undirected graph of $G$.} Since $G$ is   connected, we have $m \ge n - 1$.   Let $M = m - n +1$.  We now show by induction on $M$, that there is a desired formation $\vdq q_d\in\hat{\mathbb U}^{n}$, which satisfies (\ref{desrel}) for all $(i, j) \in E$. 		
			
			(i)  {We first show this for $M = 0, i.e., m = n-1$.}   Then $G$ is a tree.
			{Without loss of generality, we pick node $1$ as the root of this tree.}
			Let $\hat q_{d_1} = \hat 1$.  Then we may determine $\hat q_{d_j}$ for all $j \not = 1$ by \eqref{desrel} from each father node to its children nodes.
			Then  equation \eqref{desrel} is satisfied {for all $(i, j) \in E$}.

			(ii) We now assume that this claim is true for $M = M_0$, and prove that it is true for $M = M_0 +1$.   {Then $M > m_0 - n + 1$.  This implies that} there is at least one cycle  $\left\{ j_1, \dots, j_k \right\}$, with $j_{k+1} = j_1$, of $G$, such that the $k$ nodes $j_1, \dots, j_k$ are all distinct and $k \ge 3$. Without loss of generality, we may assume that $(j_k, j_1) \in E$.  Delete the  edge  $(j_k, j_1)$ from $E$.  Let $E_0 = E \setminus \{ (j_k, j_1) \}$, and $G_0 = (V, E_0)$.  Then $G_0$ is still directed and {connected}, and the desired relative configurations $\left\{ \hat q_{d_{ij}} : (i, j) \in E_0 \right\}$ satisfy  (\ref{cycle}) for all cycles in $G_0$.  By our induction assumption, there is a desired formation $\vdq q_d\in\hat{\mathbb U}^{n\times 1}$ such that (\ref{desrel}) holds
			for all $(i, j) \in E_0$.  Now,  by this and (\ref{cycle}) for this cycle $\left\{ j_1, \dots, j_k \right\}$,  (\ref{desrel}) also holds for $(j_1, j_k)$.  This proves this claim for $M = M_0+1$ in this case.

			Suppose that  ${\vdq q}_0$ is a desired formation satisfying (\ref{desrel}), and $\hat c \in \hat{\mathbb U}$.  Then $\hat c {\vdq q}_0$ also satisfies (\ref{desrel}).   We may prove the other side of the last claim by induction as above.
			
			{The proof is completed.}
		\end{proof}

		The set in} (\ref{alldesiredformation}) means  that while the $n$ rigid bodies converge asymptotically to the desired relative formation, they may still have rotations and translations as a cohesive whole as in \cite{LWHF14}, or they have a lead which moves following some trajectories and the other rigid bodies follow that lead as in \cite{WYL12}.
	
	
	{The equivalent condition in \eqref{cycle} requires enumerating all cycles, which may be expensive for a large scale graph.
		In the following theorem, we show another equivalent condition based on the zero eigenvalue of the Laplacian matrix and its corresponding eigenvectors.  {For the knowledge of eigenvalues of a dual quaternion matrix, see \cite{QL23}.}}

	\begin{Thm} \label{similar} Suppose that we have a UDQDG $\Phi = (G, {\udqset}, \varphi)$, where  $G = (V, E)$, {and} a {desired relative configuration}  {scheme} $\left\{ \hat q_{d_{ij}} \equiv \varphi(i, j): (i, j) \in E \right\}$ on $G$.   Assume that $G$ has a directed spanning tree.  Let the dual quaternion Laplacian {$\dqm L$} be constructed as above.  Then the {desired relative configuration} is reasonable, if and only if  $\dqm L$ has a unique zero eigenvalue, with an eigenvector  $\bar{\vdq q}_d=[{\dq q}^*_{d_1},\cdots,{\dq q}^*_{d_n}]^\top\in \hat{\mathbb U}^{n}$ satisfying 		
		\begin{equation} \label{null}
			\dqm L \bar{\vdq q}_d = \vdq 0,
		\end{equation}
		and
		\begin{equation} \label{LL1}
			\dqm Q {\dqm L}\dqm Q^*=L= D-A,
		\end{equation}
		where  {$\dqm Q=\mathrm{diag}(\vdq q_d)$},  $D$, $A$, and $L$ are the  {out-degree} matrix, the  {unweighted} adjacency matrix, and the unweighted Laplacian matrix of $G$, respectively.
	\end{Thm}
	\begin{proof}
		Suppose that ${\vdq q}_d$ is a desired formation.  Then it satisfies (\ref{desrel}) {and} the left {hand side} of the $i$th equation of (\ref{null}) satisfies
		$$d_i\hat q_{d_i}^* - \sum_{(i, j) \in E} \hat q_{d_{ij}}\hat q_{d_j}^*= d_i\hat q_{d_i}^* - \sum_{(i, j) \in E} \hat q_{d_i}^*\hat q_{d_j}\hat q_{d_j}^* = 0.$$
		Thus, (\ref{null}) {holds true}.
		{Equation (\ref{LL1}) follows directly from (\ref{desrel}).
			It follows from the fact that zero is a unique eigenvalue of $L$ and (\ref{LL1}) that   zero is also a unique eigenvalue of $\dqm L$.
		}

		On the other hand, suppose that $\vdq q_d \in {\hat{\mathbb U}^{n}}$  satisfies (\ref{null}) and  {\eqref{LL1}}.   We now need to show that it satisfies (\ref{desrel}).
		Then, by (\ref{LL1}), we have			 $ {\dqm L}=	\dqm Q ^* L\dqm Q$ and
		{\[ \hat q_{d_{ij}} = \hat q_{i}^*\hat q_j.\]}
		The sufficient condition holds  from {the definition}.
	\end{proof}
	
	This theorem establishes the relation between desired formations and the null space of the dual quaternion Laplacian,  and the relation between the dual quaternion Laplacian of a UDQDG $\Phi = (G, \udqset, \varphi)$, and the ordinary Laplacian of the directed graph $G$.

	\section{A Direct Method to Determine if a Desired Relative Configuration Scheme is Reasonable or Not}\label{sec:direct_method}
	
	Suppose that we have a  {connected directed} graph $G = (V, E)$, and a
	desired relative configuration scheme $\left\{ \hat q_{d_{ij}} : (i, j) \in E \right\}$.   A practical question is: Whether such a desired relative configuration scheme is  reasonable or not?
	Based upon the discussion in the last section, we have the following procedure to answer the above question.
	
	{Step 1.} If there are arcs $(i, j), (j, i) \in E$, check if (\ref{symmetry}) is satisfied or not.   If  (\ref{symmetry}) is not satisfied for such a pair of arcs, then the desired relative configuration scheme is  not reasonable.
	
	{Step 2.} Construct the dual quaternion Laplacian matrix $\hat L$ by (\ref{Laplaican_mat}).
	
	{Step 3.} Solve
	\begin{equation} \label{null1}
		\dqm L {\vdq x} = \vdq 0, \text{ and } \vdq x \in \hat{\mathbb U}^{n}.
	\end{equation}
	If  such a solution  exists, then by Theorem~\ref{similar},  we {denote $\vdq q_d \equiv \bar{\vdq x}\in\udqset^n$ and}  check whether \eqref{LL1} holds. If so,
	{$\vdq q_d$} is a desired formation satisfying (\ref{desrel}) and the desired relative configuration scheme is  reasonable.   Otherwise, it is not.

	We may write $\dqm L = L_s + L_d\epsilon$ and $\vdq x = \vx_s + \vx_d\epsilon$. Then (\ref{null1}) is equivalent to two quaternion equations
	\begin{equation} \label{L_s}
		L_s \vx_s  = {\0},
	\end{equation}
	where $\vx_s \in {\mathbb U}^{n}$,
	and
	\begin{equation} \label{L_d}
		L_s \vx_d  {=- L_d\vx_s},
	\end{equation}
	where {$x_{sj}x_{dj}^* + x_{dj}x_{sj}^* = 0$} for $j = 1, \dots, n$.

	However, solving \eqref{L_s} directly  may produce a zero solution which does not  satisfy $\vx_s \in {\mathbb U}^{n}$, even though  {the desired relative configuration scheme} is reasonable.
	The situation may be even more complicated when the dimension of the null space of $L_s$ is  greater than 1.
	In this case, we must exploit the  null space of $L_s$ to see whether $\vx_s \in {\mathbb U}^{n}$ exists.
	In the following, we make the following assumption.
	{For connected dual number or complex digraphs, we have $\mathrm{rank}(L_s)= n-1$. This assumption holds directly. However, for dual quaternion directed graphs, it is not clear yet.}
	
	\begin{Ass}\label{ass:rk>n-1}
		Given a  {connected directed} graph $G = (V, E)$, and a
		desired relative configuration scheme $\left\{ \hat q_{d_{ij}} : (i, j) \in E \right\}$. 	Assume the standard part of the Laplacian matrix satisfies  $\mathrm{rank}(L_s)\ge n-1$.
	\end{Ass}
	
	Under Assumption~\ref{ass:rk>n-1}, either \eqref{L_s} only has zero solution or all solutions of  \eqref{L_s}  may be denoted by
	\begin{equation}\label{sol_set}
		[\vx_s]=\{\vy: \vy= \vx_sq,\ q\in\mathbb Q\},
	\end{equation}
	where $\vx_s$ is a nonzero solution of  \eqref{L_s}.
	Similarly, when  Assumption~\ref{ass:rk>n-1} holds and $\vx_s$ satisfies \eqref{L_s}, either \eqref{L_d} has no solution at all or all solutions of  \eqref{L_d}    may be denoted by
	\begin{equation}\label{sol_set_d}
		[\vx_d]=\{\vy: \vy=\vx_d- \vx_sq,\ q\in\mathbb Q\},
	\end{equation}
	where $\vx_d$ is a  solution of  \eqref{L_d}.

	Different from the {real field, for quaternion rings,} the nonzero solution of \eqref{L_s} is not unique even if the dimension of  the null space of $L_s$ is  equal to 1. 	
	In fact, for any $\vx_s \in {\mathbb U}^{n}$  satisfying $L_s \vx_s  = {\0}$ and any $q\in \mathbb U$,   $\vx_s q\in {\mathbb U}^{n}$ is also a solution of \eqref{L_s}.
	Thus, we fix $x_{1s}=1$ and solve the reduced system
	\begin{equation} \label{L_s_reduced}
		L_{2s} \vx_{2s}  = -L_{1s},
	\end{equation}
	where $L_{1s}\in\mathbb Q^{n}$ and $L_{2s}\in\mathbb Q^{n\times (n-1)}$ are the first and  remaining columns of $L_s$, respectively, and $\vx_{2s}\in\mathbb Q^{n-1}$ is the last $n-1$ elements of $\vx_s$.
	If \eqref{L_s_reduced} does not admit any solutions, then \eqref{L_s} only has zero solution  and  {the configuration scheme} is not reasonable.
	{Otherwise, if} \eqref{L_s_reduced} has a nonzero solution $\vx_s$, we still have to check whether $\vx_s \in {\mathbb U}^{n}${: If} $\vx_s \in {\mathbb U}^{n}$, then we go to next step and compute $\vx_d${; Otherwise,} 	{by  the expression in \eqref{sol_set} and $x_{1s}=1$, we should choose $q\in\mathbb U$, which does not change the magnitude of  the entries. Thus, we conclude that $\vy \notin {\mathbb U}^{n}$ for any $\vy \in [\vx_s]$ and} the configuration  {scheme} is not reasonable.

	Similarly, the solution of \eqref{L_d} is not unique when \eqref{L_s} has nonzero solutions. Thus, we fix $x_{1d}=0$ and solve the reduced system
	\begin{equation} \label{L_d_reduced}
		L_{2s} \vx_{2d}  =- L_d\vx_s,
	\end{equation}
	where $\vx_{2d}\in\mathbb Q^{n-1}$ is the last $n-1$ entries of $\vx_d$.
	{In fact, for any solution $\vy_d$ of \eqref{L_d}, we can set $\vx_d=\vy_d-\vx_sy_{1d}$ and derive $x_{1d}=0$.}
	If \eqref{L_d_reduced} is  solvable and   {$x_{sj}x_{dj}^* + x_{dj}x_{sj}^* = 0$} for $j = 1, \dots, n$, then the configuration  {scheme} is reasonable. Otherwise,   the configuration {scheme} is not reasonable  since for any solution $\vy\in [\vx_d]$, there is
	\[x_{sj}y_j^* + y_jx_{sj}^* = x_{sj}x_{dj}^* + x_{dj}x_{sj}^* - x_{sj}(q+q^*)x_{sj}^* = x_{sj}x_{dj}^* + x_{dj}x_{sj}^*- 2q_0,\]
	where $q_0$ is the standard part of $q$.
	This {combining} with $x_{1d}=0$ derives $q_0=0$.
	Thus we conclude that for any solution $\vy\in [\vx_d]$, there is  $x_{sj}y_j^* + y_jx_{sj}^* = x_{sj}x_{dj}^* + x_{dj}x_{sj}^*$ for $i=1,\dots,n$ and  the configuration  {scheme} is not reasonable.

	We summarize the whole procedure in
	{Algorithm~\ref{alg:is_resonable}.}
	\begin{algorithm}[t]
		\caption{A  direct  method to determine if a desired relative configuration scheme is reasonable or not}\label{alg:is_resonable}
		\begin{algorithmic}[1]
			\Require  a directed and connected graph $G = (V, E)$, and a
			desired relative configuration scheme $\left\{ \hat q_{d_{ij}} : (i, j) \in E \right\}$.
			\State  If there are arcs $(i, j), (j, i) \in E$, check if (\ref{symmetry}) is satisfied or not.   If  (\ref{symmetry}) is not satisfied for such a pair of arcs, then \textsc{Flag}=0 and stop.
			\State Construct the Laplacian matrix $L\in\mathbb Q^{n\times n}$ and solve \eqref{L_s_reduced}.
			\If{\eqref{L_s_reduced} is solvable and $\vx_s\in\mathbb U^n$}
			\State Solve   \eqref{L_d_reduced}.
			\If{\eqref{L_d_reduced} is solvable and {$\vdq q_d \equiv \bar{\vdq x}\in\udqset^n$}  satisfies \eqref{LL1}}
			\State  \textsc{Flag}=1 and stop.
			\Else
			\State  \textsc{Flag}=0 and stop.
			\EndIf
			\Else
			\State  \textsc{Flag}=0 and stop.
			\EndIf
			\State \textbf{Output:} \textsc{Flag}.
		\end{algorithmic}
	\end{algorithm}

	Let us revisit the examples in {Example~\ref{ex:2.1}.} For the example of directed tree {$G^{(1)}=(V^{(1)},E^{(1)})$}, we have
	\begin{equation*}
		L_{1s}^{(1)} =\begin{bmatrix} 0\\   -q_{d_{21s}} \\ 	  -q_{d_{31s}}\end{bmatrix}, \quad  	  L_{2s}^{(1)} = \begin{bmatrix}
			0 & 0 \\
			1 & 0\\
			0 & 1
		\end{bmatrix}, \text{ and } L_d^{(1)} =  	  \begin{bmatrix}
			0 & 0 &0 \\
			-q_{d_{21d}} & 0 & 0\\
			-q_{d_{31d}} &0 & 0
		\end{bmatrix}.
	\end{equation*}
	Thus, by solving \eqref{L_s_reduced} and \eqref{L_d_reduced}, we have $\vx_s^{(1)} = [1,  q_{d_{21s}},  q_{d_{31s}}]^\top$ and  $\vx_d^{(1)} = [0,  q_{d_{21d}},  q_{d_{31d}}]^\top$, respectively. Namely, $\vx^{(1)} = [1, \hat q_{d_{21}}, \hat q_{d_{31}}]^\top\in\hat{\mathbb U}^3$ is a solution of \eqref{null1} and the desired relative configuration scheme is reasonable for any $\hat q_{d_{21}}, \hat q_{d_{31}} \in\hat{\mathbb U}$.

	For the example of three points directed cycle {$G^{(2)}=(V^{(2)},E^{(2)})$}, we have
	\begin{equation*}
		L_{1s}^{(2)} =\begin{bmatrix} 1\\   -q_{d_{21s}} \\ 	 0\end{bmatrix}, \quad  	  L_{2s}^{(2)} = \begin{bmatrix}
			0 & -q_{d_{13s}} \\
			1 & 0\\
			-q_{d_{32s}} & 1
		\end{bmatrix}, \text{ and } L_d^{(2)} =  	  \begin{bmatrix}
			0 & 0 &-q_{d_{13d}}  \\
			-q_{d_{21d}} & 0 & 0\\
			0 &-q_{d_{32d}}  & 0
		\end{bmatrix}.
	\end{equation*}
	Thus,  system \eqref{L_s_reduced} is consistent  only if $q_{d_{13s}}^*= q_{d_{32s}}q_{d_{21s}}$. In this case, we have $\vx_s^{(2)} = [1,  q_{d_{21s}},  q_{d_{13s}}^*]^\top$.   System \eqref{L_d_reduced} is  consistent  only if $q_{d_{13d}}^*= q_{d_{32s}}q_{d_{21d}}+q_{d_{32d}}q_{d_{21s}}$. In this case, we have $\vx_d^{(2)} = [0,  q_{d_{21d}},  q_{d_{13d}}^*]^\top$.   Namely,   the desired relative configuration scheme is reasonable only if $\hat q_{d_{13}}^*=\hat q_{d_{32}}\hat q_{d_{21}}$, which has a solution $\vx^{(2)} = [1, \hat q_{d_{21}}, \hat q_{d_{13}}^*]^\top\in\hat{\mathbb U}^3$ for \eqref{null1}.
	{We could verify that $\vdq q_d \equiv \vdq x^{(2)}$  satisfies \eqref{LL1}.}

	In practice, we may convert \eqref{L_s_reduced} {and \eqref{L_d_reduced}}  to real or complex equations,  whose dimension is expanded to four times or twice of
	the original dimension.   Then they are surely solvable.
	For instance,  any quaternion linear equation $A\vx=\vb$ with $A\in\mathbb Q^{m\times n}$,  $\vb\in\mathbb Q^{m}$ can be equivalently rewritten to a real matrix equation
	\begin{equation} \label{R_L_s}
		\R(A) \R_v(\vx)  = \R_v(\vb),
	\end{equation}
	where $\R(\cdot)$ is a linear homeomorphic mapping from quaternion matrices (or vectors) to their real counterpart and $\R_v(\cdot)$ denotes the first column of $\R(\cdot)$.
	Specifically, let $A=[A_0, A_1, A_2, A_3]$, then
	\[
	\R(A) = \begin{bmatrix}
		A_0 & -A_1 & -A_2 & -A_3 \\
		A_1 & A_0 & -A_3 & A_2 \\
		A_2 & A_3 & A_0 & -A_1 \\
		A_3 & -A_2 & A_1 & A_0	
	\end{bmatrix} \text{ and } 	\R_v(A) = \begin{bmatrix}
		A_0\\
		A_1 \\
		A_2\\
		A_3	
	\end{bmatrix}\]
	Several structure-preserving   algorithms such as \cite{DLLWZ24, JN21, LWZ23} 
	can be used here to solve \eqref{R_L_s}  efficiently.
	
	\section{A Gain Graph Method to Determine if a Unit Weighted Directed Graph is Balanced or Not}\label{sec:gain_graph_method_balance}
	
	The adjacency and Laplacian matrices of a UDQDG are non-Hermitian.   Their spectral structures are complicated and not well-explored yet.  Here, we propose a gain graph method to convert the balance problem of a UDQDG to the balance problem of a dual quaternion unit gain graph whose adjacency and Laplacian matrices are Hermitian.
	
	\medskip
	
	This method also works for general unit weighted directed graphs.    A general unit weighted directed graph has the form $\Phi = (G, \Omega, \varphi)$, where $\Omega$ is the weight group, which is a mathematical group, all of whose elements have unit $1$.  When $\Omega = \udqset$, we have a UDQDG.   When $\Omega = \{-1, 1 \}$, $\Phi$ is a real unit directed graph.   Real unit directed graphs and their applications have been studied by  Kramer and Palowitch \cite{KP87}, Yang, Shah and Xiao \cite{YSX10}, with the name signed directed graphs, and Li, Yuan, Wu and Lu \cite{LYWL18} with the name directed signed graphs.   Applications of real unit directed graphs include process systems analysis, complex network analysis and data mining.   Wissing and van Dam  \cite{WvD22}, Ko and Kim \cite{KK23} studied some special complex unit weighted directed graphs.
	
	Suppose that we have a unit weighted directed graph  $\Phi = (G, \Omega, \varphi)$, where $G=(V, E)$ and $\varphi(i, j) = \hat q_{d_{ij}}$ for $(i, j) \in E$.   We may first check whether there are arcs $(i, j), (j, i) \in E$, and if there are such arc $(i, j) \in E$,  whether (\ref{symmetry}) is satisfied or not.   If  (\ref{symmetry}) is not satisfied for such a pair of arcs, then  $\Phi$ is not reasonable (balanced).
	
	\medskip
	
	If $\Phi$ passes the above check, then we may construct $G_1=(V, E_1)$ such that for any $(i, j) \in E$, we have $(i, j) \in E_1$ and $(j, i) \in E_1$, and construct $\Phi_1 = (G_1, \Omega, \varphi_1)$ such that for any $(i, j) \in E$, $\varphi_1(i, j) = \varphi(i, j)$ and $\varphi_1(j, i) = \left(\varphi(i, j)\right)^{-1}$.  Then $\Phi_1$ is a unit gain graph, and $\Phi$ is reasonable (balanced) if and only if $\Phi_1$ is reasonable (balanced).  We may use the method provided in \cite{QC24} to test whether $\Phi_1$ is reasonable (balanced) or not.   We call this method the gain graph method.
	
	For a unit weighted directed graph $\Phi = (G, \Omega, \varphi)$, where $G = (V, E)$, we say that a function $\theta : V \to \Omega$ is a potential function of $\Phi$ if for every $(i, j) \in E$, we have
	$$\varphi(i, j) = \theta(i)^{-1}\theta(j).$$
	Then we may easily prove the following proposition.

	\begin{Prop}
		Suppose the unit weighted directed graph $\Phi$ passes the above check, and a unit gain graph $\Phi_1$ is constructed.    Then a function $\theta : V \to \Omega$ is a potential function of $\Phi$ if and only if it is a potential function of $\Phi_1$.
	\end{Prop}
	
	Suppose that we have a directed and connected graph $G = (V, E)$, and a
	desired relative configuration scheme $\left\{ \hat q_{d_{ij}} : (i, j) \in E \right\}$.   Again, we wish to answer the following question: Whether such a desired relative configuration scheme is  reasonable or not?    The gain graph method can be described as follows.
	
	{Step 1.} If there are arcs $(i, j), (j, i) \in E$, check if (\ref{symmetry}) is satisfied or not.   If  (\ref{symmetry}) is not satisfied for such a pair of arcs, then the desired relative configuration scheme is  not reasonable.
	
	{Step 2.} {Construct} $G_1=(V, E_1)$ such that for any $(i, j) \in E$, we have $(i, j) \in E_1$ and $(j, i) \in E_1$, and construct $\Phi_1 = (G_1, \Omega, \varphi_1)$ such that for any $(i, j) \in E$, $\varphi_1(i, j) = \varphi(i, j)$ and $\varphi_1(j, i) = \left(\varphi(i, j)\right)^{-1}$.   Construct the dual quaternion Laplacian matrix $\hat L_1$ correspondingly.
	
	{Step 3.} Solve
	\begin{equation} \label{null2}
		\dqm L_1 {\vdq x} = \vdq 0, \text{ and } \vdq x \in \hat{\mathbb U}^{n},
	\end{equation}
	{or equivalently, compute the eigenvector of $\dqm L_1$ corresponding to the zero eigenvalue, denoted by $\vdq x$.}
	If  such a solution  exists, then we  check whether we have
	$$\dqm Q {\dqm L_1}\dqm Q^*=L_1= D-A_1,$$
	where  {$\dqm Q=\mathrm{diag}(\bar{\vdq x})$},  $D$, $A_1$, and $L_1$ are the  {out-degree} matrix, the  {unweighted} adjacency matrix, and the unweighted Laplacian matrix of $G_1$, respectively.
	If so,
	then the unit gain graph $\Phi_1$ and the unit weighted directed graph $\Phi$ are balanced.   Otherwise, they are not.
	
	This method can be regarded as an alternative approach for testing the unit weighted directed graph $\Phi$ is balanced or not.   The difference of these two methods is that the coefficient matrix in (\ref{null2}) is Hermitian, while the coefficient matrix in (\ref{null1}) is non-Hermitian.    It depends if there are algorithms which can use the advantage that the the coefficient matrix is Hermitian, to solve these equations.
	{In addition, for the grain graph method, we may compute the the eigenvector of $\dqm L_1$ corresponding to the zero eigenvalue.}

	For {the balance problem of} a non-unit weighted directed graph, there {may be no} advantage to convert it to a non-unit gain graph, as the coefficient matrix in (\ref{null2}) is still non-Hermitian in this case.   Thus, in this case, maybe the direct method is {better}.   Also, the out-degree needs to be carefully defined.   Thus, we study the {balance problem} of general non-unit weighted directed graphs in {the next} section.
	
	{\section{General Non-Unit Weighted Directed Graphs}}
	
	
	Suppose that we have a  {loopless} directed graph $G = (V, E)$, a multiplication group $\Omega \subset \dqset$ and a {\bf weight function} $\varphi : E \to \Omega$.  Here, $V$ is the vertex set and $E$ is the arc set.
	
	Let $i, j \in V$.  If $(i, j) \not \in E$, then we have ${\varphi}(i, j) = 0$.  If $(i, j) \in E$, then we have ${\varphi}(i, j) = \hat q_{ij} \in \Omega$.  We call ${\varphi}(i, j)$ the {\bf weight} of the
	arc $(i, j)$,  {and} $\Omega$ the {\bf weight group}  {of $\Phi$}.   Here, $\Omega$ contains non-unit elements.

	The  {triple} $\Phi = (G, \Omega, \varphi)$ is  {called} a (non-unit) {{\bf weighted directed graph (WDG)}}.
	
	{\bf Example 1} If $\Omega = {\mathbb R} \setminus \{ 0 \}$, then we have
	{\bf real weighted directed graphs}.  Such a weighted directed graph has a long history.   For example, in 1991, H. Schneider and M.H. Schneider \cite{SS91} studied some max-min properties of such real weighted
	directed graphs.   Actually, if we regard such real weight as a flow value and add a source node and a sink node to the real weighted directed graph to make the in-flow and out-flow at each other node to be balanced, then we have
	a network flow problem, and the famous max-flow-min-cut theorem is valid here.  In 2017, Ahmadizadeh, Shames, Martin and Ne\v{s}ic \cite{ASMN17} studied eigenvalues of the Laplacian matrix of such a real weighted directed graph with the name directed signed graphs.    Also in 2017, Jiang, Zhang and Chen \cite{JZC17}
	studied such a real weighted directed graph with the name signed directed graphs.
	
	{\bf Example 2} If $\Omega = {\mathbb R}_{++}$, the set of positive numbers, then we have {\bf nonnegative weighted directed graphs}.   Nonnegative weighted directed graphs have found broad applications in control engineering \cite{OFM07}.
	
	{\bf Example 3} If $\Omega$ is the set of nonzero complex numbers, then we have {\bf complex WDG}s.   Notably, complex WDGs are shown to be very useful in multi-agent formation control.   Dong and Lin \cite{DL16} studied such non-unit weighted directed graphs. Notably, complex WDGs are shown to be very useful in multi-agent formation control {\cite{LWCFH15, LH15}}.
	Besides applications in formation control, networks with complex weights also arise in fields as diverse as quantum information, 	quantum chemistry, electrodynamics, rheology, and machine learning \cite{BP24}.
	
	{\bf Example 4} If $\Omega$ is the set of appreciable dual quaternion numbers, then we have {\bf dual quaternion weighted directed graphs}.    A dual quaternion number has the form $\hat q = q_s + q_d \epsilon$, where $q_s$ and $q_d$ are two quaternion numbers, $q_s$ is called the standard part of $\hat q$, $q_d$ is called the dual part of $\hat q$. If $q_s \not = 0$, then $\hat q$ is called appreciable.   A dual quaternion number is invertible if and only if it is  appreciable.  Thus,   the set of appreciable dual quaternion numbers forms a mathematical group.   The other examples of non-unit weighted directed graphs, listed earlier, are special cases of dual quaternion weighted directed graphs.   Thus, we may study dual quaternion weighted directed graphs as the prototype of general non-unit weighted directed graphs.
	
	
	A {\bf walk} $W$ in a  {WDG} $\Phi$ is defined in this way. It consists of vertices $i_1, i_2, \dots, i_k$.   For $j = 1, \dots, k-1$, either $(i_j, i_{j+1}) \in W$ or $(i_{j+1}, i_j) \in W$.  If
	$(i_j, i_{j+1}) \in W$, we say that the arc $(i_j, i_{j+1})$ is a {\bf forward arc} of $W$ and let the shadow element $\hat p_{i_ji_{j+1}} = \hat q_{i_ji_{j+1}}$.  If
	$(i_{j+1}, i_j) \in W$, we say that the arc $(i_{j+1}, i_j)$ is a {\bf backward arc} of $W$ and let the shadow element $\hat p_{i_ji_{j+1}} = \hat q_{i_{j+1}i_j}^{-1}$. The {\bf weight} of the walk $W$ is
	\begin{equation}
		\varphi(W) = \prod_{j=1}^{k-1} \hat p_{i_ji_{j+1}}.
	\end{equation}
	The walk $W$ is {\bf neutral} if $\varphi(W)$ {is a real positive number}.
	The WDG $\Phi$  is called {\bf balanced} if every cycle $C \subseteq  {E}$ is neutral.

	The definition of  {balanced} here is different from that of \cite{DL16} since only forward arcs are considered  in \cite{DL16}.
	We exploit the difference by the following example, which is a slight modification of Example 4  {of} \cite{DL16}.
	\medskip
	
	\begin{example}\label{ex:5.1}
		Consider two complex WDGs  in Fig.~\ref{fig:balance}.
		The complex  {WDG} on  the  left is the same with {Example 4} of  \cite{DL16}. It is unbalanced  {under our definition} since $\phi(C)=-1$ for the cycle $C=\{1,2,3,1\}$.  	The complex  {WDG} on  the  right is balanced.
	\end{example}
	We would like to point out that 	the complex  {WDG} on  the  left of Fig.~\ref{fig:balance} is balanced  {under} the definition of  {balanced} {in} \cite{DL16}.
	
	\begin{figure}\label{fig:balance}
		\begin{center}
			\includegraphics[width=0.3\linewidth]{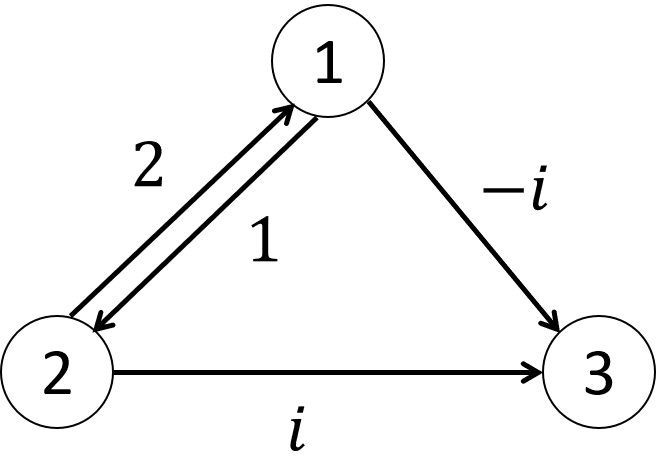} \qquad
			\includegraphics[width=0.3\linewidth]{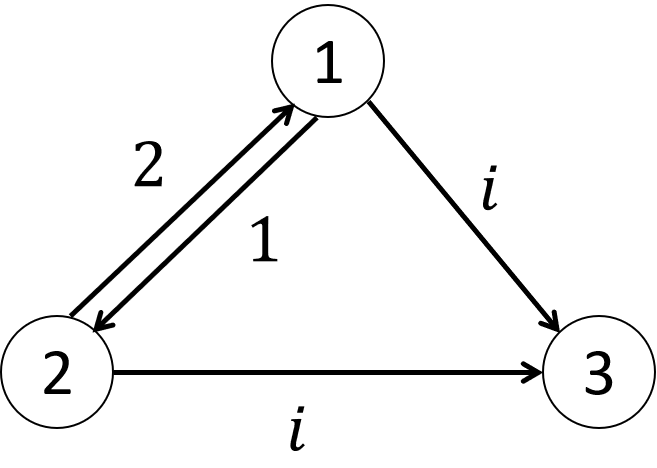}
		\end{center}
		\caption{Two  complex WDGs. Left: A unbalanced graph. Right: A balanced graph.}
	\end{figure}

	A {\bf potential function} is a function $\theta : V \to \Omega$ such that for every arc $(i, j) \in E$,
	\begin{equation} \label{potential}
		\varphi(i, j) = \theta(i)^{-1}\theta(j){c_{ij},}
	\end{equation}
	{for a certain positive real number $c_{ij}$.}
	
	A {\bf switching function} is a function $\zeta : V \to \Omega$ that switches the directed gain graph $\Phi = (G, \Omega, \varphi)$ to $\Phi^\zeta = (G, \Omega, \varphi^\zeta)$, where
	\begin{equation}  \label{switching}
		\varphi^\zeta(i, j) = (\zeta(i))^{-1}\varphi(i, j)\zeta(j).
	\end{equation}
	Thus, in this case, $\Phi$ and $\Phi^\zeta$ are switching equivalent, denoted by $\Phi \sim \Phi^\zeta$.
	Denote $\bf 1$ as a {weight function} such that for any $(i, j) \in E$, we have ${\bf 1}(i, j) \equiv 1$.   Then we have the following theorem.
	

	\begin{Thm}
		Let $\Phi = (G, \Omega, \varphi)$ be a {WDG}, where $G = (V, E)$ is a connected directed graph.   Then the following statements are equivalent.
		
		(i) $\Phi$ is balanced.
		
		(ii) $\Phi \sim (G, {\Omega, \varphi^\zeta})$, where ${\varphi^\zeta}(i,j) = |\varphi_s(i,j)|$ for all $(i,j)\in E$.

		(iii)  {$\Phi$} has a potential function {$\theta$}.
	\end{Thm}
	
	\begin{proof}
		$(iii)\Rightarrow (ii)$
		Suppose that there is a  potential function $\theta : V \to \Omega$ and positive real numbers $c_{ij}$ such that $\varphi(i, j) = \theta(i)^{-1}\theta(j)c_{ij}$  for every arc $(i, j) \in E$.
		Since the dual part of $c_{ij}$ is zero, we have
		\begin{equation}\label{cij_DQ}
			c_{ij} = \frac{ |\varphi_s(i,j)| |\theta_s(i)|}{|\theta_s(j)|},
		\end{equation}
		where $ \varphi_s(i,j)$ and $\theta_s(i)$ are  the  standard parts of $\varphi(i, j)$ and $\theta(i)$, respectively.
		For $i=1,\dots,n$, let
		\[\zeta(i) ={|\theta_s(i)|} \theta(i)^{-1}.\]
		Then we have
		\begin{equation}\label{poten_DQ}
			\varphi^\zeta(i, j) = {\frac{|\theta_s(j)|}{|\theta_s(i)|}} \theta(i)\varphi(i, j)\theta(j)^{-1}={\frac{|\theta_s(j)|}{|\theta_s(i)|}} c_{ij} =|\varphi_s(i,j)|.
		\end{equation}

		{$(ii)\Rightarrow (i)$  It follows from  \eqref{poten_DQ}  that $\varphi(i,j)=|\varphi_s(i,j)|\zeta(i) \zeta(j)^{-1}$.
			Given any cycle $C=\{i_1,\dots,i_k,i_1\}$,   the weight of  $C$ is equal to
			\[\varphi(C) = \hat p_{i_1i_2}\dots \hat p_{i_ki_1} = |\left(p_{i_1i_2}\right)_s|\dots  |\left(p_{i_ki_1}\right)_s|,\]
			where $(p_{ij})_s$ is the standard part of the shadow element $\hat p_{ij}$.
			This  implies that all cycles are {neutral} and $\Phi$ is balanced.}

		$(i)\Rightarrow (iii)$   Suppose that all cycles are {neutral}.   Let $m$ be the number of edges of the underlying undirected graph of $G$. Since $G$ is   connected, we have $m \ge n - 1$.
		Let $M = m - n +1$.  We now show by induction on $M$  that there is a   potential function $\theta : V \to \Omega$ and positive real numbers $c_{ij}$ such that $\varphi(i, j) = \theta(i)^{-1}\theta(j)c_{ij}$  for every arc $(i, j) \in E$.
		
		We first show this for $M = 0, i.e., m = n-1$.   Then {the underlying graph of} $G$ is a tree.
		Without loss of generality, we pick node $1$ as the root of {the underlying tree.}
		Let $\theta(1) =1$.  Then for all $j \not = 1$,  we may  {let $c_{ij}=1$   and  $\theta(j)=\theta(i)\varphi(i,j)$  if $(i,j)\in E$ and $\theta(j)=\theta(i)\varphi(j,i)^{-1}$  if $(j,i)\in E$  from each father node $i$ to its children node $j$.
			
			We now assume that this claim is true for $M = M_0$, and prove that it is true for $M = M_0 +1$.   This implies that there is at least one cycle  $\left\{ j_1, \dots, j_k \right\}$, with $j_{k+1} = j_1$, of $G$, such that the $k$ nodes $j_1, \dots, j_k$ are all distinct and $k \ge 3$. We may assume that $(j_k, j_1) \in E$.  Delete the  edge  $(j_k, j_1)$ from $E$.  Let $E_0 = E \setminus \{ (j_k, j_1) \}$, and $G_0 = (V, E_0)$.  Then $G_0$ is still directed and {connected}, and the  weight of cycles in  $G_0$ are positive.  By our induction assumption, there is a potential function $\theta$  and positive real numbers $c_{ij}$  such that \eqref{potential} holds
			for all $(i, j) \in E_0$.
			Let $c_{j_kj_1}=\theta(j_1)^{-1}\theta(j_k)\varphi(j_k,j_1)$.
			Then   \eqref{potential} also holds for $(j_k, j_1)$.  This proves the claim for $M = M_0+1$ in this case.
		}

		This completes the proof.
	\end{proof}
	
	{
		Denote $\Phi=(G,\Omega,\varphi)$ as	the complex  WDG  on  the  left of Fig.~\ref{fig:balance}. Then  We may check that $\theta(1)=1, \theta(2)=1, \theta(3)=\ii$ is a potential function for $\Phi$  with $c_{21}=2$, $c_{12}=c_{13}=c_{23}=1$.
		Furthermore, let $\zeta(1)=\zeta(2)=1$ and $\zeta(3)=-\ii$. Then
		$\varphi^{\zeta}(2,1)=2$ and  $\varphi^{\zeta}(1,2)=\varphi^{\zeta}(1,3)=\varphi^{\zeta}(2,3)=1$.}
	
	{We may also extend Theorem~\ref{similar} to general WDGs.}
	
	Let ${\hat A}=(\hat a_{ij})$, where $\hat a_{ij} = {\varphi(i, j)}$ of {$(i, j) \in V\times V$}.
	The matrix ${\hat A} = ({\hat a_{ij}})$ is called the {\bf adjacency matrix} of $\Phi$.
	
	The out-degree ${d_i}$ of a vertex $i$ is defined as
	\begin{equation} \label{outdegree}
		{d_i} = \sum_{j\in E_i}  \left|{\varphi_s({i,j})}\right|,
	\end{equation}
	{where $E_i=\left\{j : (i, j) \in E \right\}$ and $\varphi_s({i,j})$ is the standard part of $\varphi({i,j})$.  Thus, the out-degree $d_i$ is a nonnegative real number.}
	Note that by (\ref{e1}),
	this definition is consistent with the definition for UDQDGs in Section 3.
	
	The {\bf Laplacian matrix} of  a  {WDG} $\Phi$ is  defined by
	\begin{equation}\label{Laplaican_mat}
		\dqm L =   {D}-\dqm A,
	\end{equation}
	where  {$D$} is a {diagonal matrix}  {whose}  $i$-th diagonal element  is equal to the {out-degree}  {$d_i$} of  the $i$-th vertex.
	
	\begin{Thm} \label{similar_DQWDG} Suppose that we have a  WDG   $\Phi = (G, \Omega, \varphi)$.    Then $\Phi$ has a potential function $\theta$ if and only if there is  {
			\begin{equation}\label{equ_y}
				\vdq \vy=[\theta(1)^{-1}|\theta_s(1)|,\dots,\theta(n)^{-1}|\theta_s(n)|]^\top
		\end{equation}}   satisfies
		\begin{equation} \label{LL1_DQ}
			\dqm Y^{-1} {\dqm L}\dqm Y=L= D-A,
		\end{equation}
		where
		$\dqm Y=\mathrm{diag}(\vdq \vy)$, $A=(a_{ij})$ is  the  {weighted} adjacency matrix with $a_{ij} = |\varphi_s(i,j)|$ for all $(i,j)\in {V\times V}$.
		
		Furthermore,  assume that $\Phi$
		{has a potential function $\theta$}. Then all eigenvalues of $L$ are right eigenvalues  {of} {$\dqm L$}, which are complex numbers, and zero is an eigenvalue of $\dqm L$ that  satisfies
		\begin{equation} \label{null_DQ}
			\dqm L \vdq \vy = \vdq 0.
		\end{equation}
		In other words, all inverse potential vectors are  {eigenvectors} of {$\dqm L$}, corresponding to the zero eigenvalue of {$\dqm L$}.
		
		If, in addition,  $G$   has a directed spanning tree, then $\theta$ is a potential function if and only if  it is invertible and  $\vdq \vy$   {defined by \eqref{equ_y} satisfies $|\vy_s|={\bf 1}$ and}  \eqref{null_DQ}.
	\end{Thm}
	\begin{proof}
		Suppose that $\Phi$ has a potential function $\theta$  {which} satisfies \eqref{potential}. Then for all $(i,j)\in E$, we have
		\[\dq y_i^{-1} \varphi(i,j)\dq y_j = {\frac{|\theta_s(j)|}{|\theta_s(i)|}} \theta(i) \theta(i)^{-1}\theta(j) c_{ij} \theta(j)^{-1} =  |\varphi_s(i,j)|,\]
		where  {$\vdq \vy$ is defined by (\ref{equ_y}), and} the last equality follows from \eqref{cij_DQ}.
		For the diagonal elements,  {we have}
		\[\dq y_i^{-1}  d_i\dq y_i = d_i, \]
		where the   equality follows from  {the fact that} $d_i$  {is a real number, which} is commutative with any dual quaternion numbers.
		This derives \eqref{LL1_DQ}.
		
		On the other hand, suppose \eqref{LL1_DQ} holds. Then we could verify that \eqref{potential} holds with $c_{ij}$ satisfies  \eqref{cij_DQ}.

		{Suppose that}  {$\lambda$} is an eigenvalue of $L$ with an eigenvector  {$\vx$}.
		From the definition, we have
		\[L \vx = \lambda \vx 	\Longleftrightarrow  	\dqm Y^{-1} {\dqm L}\dqm Y \vx=  \lambda \vx
		\Longleftrightarrow   \hat L \dqm Y \vx= \dq Y\vx \lambda
		\Longleftrightarrow   \hat L   \vdq z=  \vdq z \lambda,
		\]
		where $\vdq z = \dq Y \vx$.  	In other words,
		all eigenvalues of $L$ are also right eigenvalues of $\dqm L$ in the sense of \cite{QL23}.
		By {the  spectral theory of $L$}, the zero eigenvalue of $L$ has an eigenvector ${\bf 1}$.  	
		Therefore, zero is also  {a} right eigenvalue of $\dqm L$ and its corresponding eigenvector is $\vdq y$.
		By \cite{QL23}, a real right eigenvalue of a dual quaternion matrix can be simply called an eigenvalue of that matrix.   Thus, zero is an eigenvalue of {$\dqm L$}.

		In addition, if $G$ has a directed spanning tree, then by Lemma~\ref{Lem:zeroeig_G}, zero is a simple eigenvalue of $L$.
		In this case,  the set of all invertible vectors  that satisfy \eqref{null_DQ} {and $|\vy_s|={\bf 1}$} is  $[\vdq y] = \vdq y \dq c$,  where $\dq c\in\udqset$.   By  \eqref{equ_y}, we could construct  the set $[\theta] = \dq c^{-1} \theta$ corresponding to $[\vdq y]$.  The sufficient condition holds  from the fact that $\dq c^{-1} \theta$ is also  a potential function.
		
		This completes the proof.
	\end{proof}

	

	{
		\section{Numerical Experiments}
		
		We   verify the balance of  directed weighted graphs, {or equivalently, the reasonableness of the desired relative configuration,}  by the directed method in Algorithm~\ref{alg:is_resonable} and the gain graph method in Section~\ref{sec:gain_graph_method_balance}.
		For both methods, we implement Steps 1-3 sequentially.
		{In} Step 3 of the gain graph method,
		{we solve \eqref{null2} following the same procedure as that of Section~\ref{sec:direct_method}.}
		If there exists a   condition  in Steps 1-3 {that} fails, then we let $Err=1$ and exit the process. Otherwise, denote ${\vdq x}\in\udqset^n$ as the vector returned by Step 3, and   $\dqm Q=\mathrm{diag}(\bar{\vdq x})$.
		We    define the residue  in \eqref{LL1} as follows,
		\begin{equation}\label{equ:err_balance}
			Err = \|\dqm Q {\dqm L}\dqm Q^*-L\|_{F^R},
		\end{equation}
		We say the gain graph is balanced if   $Err$ is less than a threshold, i.e.,  $10^{-8}$ in our numerical experiments.

		We use 	``direct'' and  ``GGM'' to denote the directed method and the gain graph method, respectively,  and report the CPU time in seconds (``CPU (s)'') for implementing Steps 1-3 in total.
		We  generate unit dual complex and quaternion directed cycles, with $E=\{(1,2),\dots,(n-1,n),(n,1)\}$.
		All these cycles are balanced, {and a smaller value of ``Err'' indicates a better outcome.}
		The numerical results {for $n\in\{10,20,50,100,200,500\}$} are shown  in  Table	\ref{table:cycle_balance}.
		
		From this table, we see that ``Err'' is always less than $10^{-12}$, which {demonstrates} that both ``direct'' and  ``GGM'' methods can verify the balance of the unit dual complex and quaternion directed cycles successfully.
		Comparing ``direct'' and  ``GGM'', we find that
		{there is no much difference.}
		
		\begin{table}[t]
			\centering
			\caption{Numerical results for verifying the {balance} of  unit dual complex and quaternion directed cycles.}\label{table:cycle_balance}
			\begin{tabular}{c|c|cccccc}
				\hline
				&$n$& 10 & 20 & 50 & 100 & 200 & 500  \\ \hline
				\multicolumn{7}{c}{Unit dual complex directed cycles} \\\hline
				Direct& CPU (s)  &9.29e$-$03 & 5.27e$-$04 & 1.21e$-$03 & 3.11e$-$03 & 1.53e$-$02 & 1.77e$-$01  \\
				&Err &	1.73e$-$15 & 4.11e$-$15 & 1.28e$-$14 & 1.21e$-$14 & 1.68e$-$14 & 4.03e$-$14 \\
				
				\hline
				GGM	&CPU (s) &
				4.98e$-$04 & 3.78e$-$04 & 1.62e$-$03 & 3.74e$-$03 & 1.52e$-$02 & 1.55e$-$01 \\
				&Err &	3.46e$-$15 & 1.15e$-$14 & 5.25e$-$14 & 7.95e$-$14 & 1.51e$-$13 & 6.61e$-$13 \\
				\hline
				
				\multicolumn{7}{c}{Unit dual quaternion directed cycles} \\\hline
				Direct& CPU (s)  & 9.78e$-$03 & 1.15e$-$02 & 2.23e$-$02 & 5.25e$-$02 & 1.86e$-$01 & 3.34e+00 \\
				&Err & 4.52e$-$15 & 4.15e$-$15 & 6.00e$-$15 & 1.10e$-$14 & 1.39e$-$14 & 2.41e$-$14 \\
				\hline
				GGM&	CPU (s) & 6.37e$-$03 & 8.95e$-$03 & 2.19e$-$02 & 4.97e$-$02 & 1.84e$-$01 & 3.71e+00 \\
				&Err & 1.23e$-$14 & 1.38e$-$14 & 3.35e$-$14 & 7.48e$-$14 & 2.13e$-$13 & 5.89e$-$13\\
				
				\hline
			\end{tabular}
		\end{table}
		
		\section{Conclusions and Future Work}
		In this paper, we studied the unit dual quaternion directed graphs arising from the formation control problem.
		Two practical numerical methods, including the direct approach and the unit grain graph method, were proposed to verify whether the desired relative configuration scheme is reasonable or not. We then generalize the results to the balance problem of general non-unit weighted directed graphs.
		
	}
	In \cite{DL16}, Dong and Lin studied complex WDGs and derived necessary and sufficient conditions for the Laplacian matrix to be singular/ nonsingular.   The contribution of Dong and Lin is that their results were established without restriction that the adjacency matrix is Hermitian.   On the other hand,  there is no theory on eigenvalues of non-Hermitian dual quaternion matrices.  Hence, it would be interesting to extend the results of Dong and Lin to general non-unit {WDGs}.  We leave this as a future research work. 	{Other problems such as  graph energy \cite{LSG12}, {main eigenvalues of graphs \cite{SY22} and inertia of graphs \cite{YQT15}} could also be considered.

		\bigskip

		\bigskip


		
		{{\bf Acknowledgment}  
			This work was partially supported by   the R\&D project of Pazhou Lab (Huangpu) (Grant no. 2023K0603),  the National Natural Science Foundation of China (Nos. 12471282 and 12131004), and the Fundamental Research Funds for the Central Universities (Grant No. YWF-22-T-204).
			
			

			{{\bf Data availability} Data will be made available on reasonable request.

				{\bf Conflict of interest} The authors declare no Conflict of interest.}

			


		\end{document}